\documentclass{ifcolog}
\usepackage{amsmath,amsthm,amssymb,amsfonts}
\usepackage{hyperref}

\usepackage{tikz-cd}   
  \usetikzlibrary{arrows,arrows.meta}
  \tikzcdset{arrow style=tikz, diagrams={>=stealth'}}
\usepackage{xcolor}    

\newcommand{\bito}{\mathrel{%
  \hspace{.1ex}
  \begin{tikzpicture}[baseline=-.57ex, line width=.125ex]
    \draw[-] (.2ex,0) --(1.5ex,0); 
    \draw[-, line width=.01ex, fill=black]
             (1.35ex,0) -- (1.84ex, .48ex)
                       -- (1.91ex ,.418ex)
                       -- (1.55ex,   0ex)
                       -- (1.91ex ,-.418ex)
                       -- (1.84ex,-.48ex)
                       -- (1.35ex,0ex);
  \end{tikzpicture}
\hspace{.1ex}}}

\newcommand{\bDiamond}{\hspace{.17ex}\text{%
  \tikz[scale=1.00, baseline=-.61ex, rounded corners=.0ex, line
width=.001ex]
    {\draw[-,fill=black] (0,-.76ex) -- (-.7ex,0) -- (0,.76ex) --
(.7ex,0);}}\kern.17ex}

\newcommand{\mo}[1]{\mathfrak{#1}}    

\newcommand{\biLan}{\mathbf{BI}} 

\newcommand{\Prop}{\textup{Prop}}   

\newcommand{\biInt}{\mathbf{BiInt}}  
\newcommand{\dom}[1]{\mbox{dom}(#1)}

\newcommand{\biItp}{\mathsf{bitp}}    
\newcommand{\nbiItp}{\mathsf{nbitp}}  

\newcommand{\uparr}{{\uparrow}}
\newcommand{\downarr}{{\downarrow}}

\renewcommand{\phi}{\varphi}
\renewcommand{\iff}{\quad\text{iff}\quad}

\newcommand{\sledom}{\Relbar\joinrel\mathrel{|}} 

\newcommand{\dbisim}{\rightharpoonup} 

\newtheorem{definition}{Definition}
\newtheorem{theorem}{Theorem}

\newtheorem{corollary}[theorem]{Corollary}
\newtheorem{lemma}[theorem]{Lemma}
\newtheorem{proposition}[theorem]{Proposition}

\definecolor{newcol1}{rgb}{1.0, 0.49, 0.0}
\definecolor{newcol0}{rgb}{0.5, 0.49, 0.0}

\title{Definable Classes of Models  and Frames \\ in Bi-intuitionistic Logic}
\titlerunning{Definable Classes of Models in Bi-Int}
\titlethanks{We are grateful to Jim de Groot with whom we discussed several of
  the ideas that appear here. We are also grateful to Katalin Bimbo for many
  editing suggestions that improved the presentation greatly.}

\addauthor{Guillermo Badia\thanks{Thanks to ...}}{University of Queensland,  Australia}
\addauthor{Tomasz Kowalski\thanks{Support from ERC HORIZON2020-MSCA-RISE project
    no.~101007627 (MOSAIC) is gratefully acknowledged.}}{Jagiellonian University, Poland}
\addauthor{Grigory Olkhovikov\thanks{Thanks to ...}}{Ruhr University Bochum,  Germany}    

\authorrunning{Badia, Kowalski, Olkhovikov}

\date{}                                           

\begin{document}

\maketitle

\begin{abstract}
The question of the expressive power of a given logical language with Kripke
relational semantics has at least two dimensions: (1) what the language can  say
about frames, and (2) what it can say about models. The  Goldblatt--Thomason
theorem provides a model-theoretic characterisation of modal axiomatisability
for elementary classes of \emph{frames} in terms of closure under taking
generated subframes, disjoint unions, bounded morphic images, and reflection of
ultrafilter extensions.  Goldblatt  also provides a similar  characterisation
for axiomatisability in intuitionistic logic of classes of \emph{models} rather
than frames. In this article we provide analogous results for bi-intuitionistic
logic, a natural expressive extension of intuitionistic logic obtained by adding
a binary connective dual to the intuitionistic implication,
introduced in the 1970s independently by Dieter Klemke  and Cecylia
Rauszer.
Together with previous
results, such as a van Benthem bisimulation characterisation theorem  and a
Lindstr\"om theorem, this provides a complete picture of the expressive power of
propositional bi-intuitionistic logic. 

\end{abstract}

\section{Introduction}\label{sec:intro}

 The result of adding the algebraic dual operation of    intuitionistic
 implication $\rightarrow$ (known as co-implication and denoted here by $\bito$)
 to intuitionistic logic is known as Heyting-Brouwer or bi-intuitionistic logic.   It is a natural model-theoretic extension of the language of intuitionistic logic in the same sense as adding a backward looking $\Diamond^{-1}$ to  modal logic with  the primitive modality $\Box$ is.  Indeed, there is a strong analogy in a  precise sense between temporal logic where we have both backward and forward looking operators and bi-intuitionistic logic (cf. \cite{w}).  In the 1970s, C. Rauszer started an intense study  of various technical aspects of both propositional and predicate bi-intuitionistic logic in a series of interesting articles spanning over a decade \cite{ra1, ra2, ra3, ra4, ra5, ra6}. 
This work has been picked up in  recent years by a number of scholars  (e.g
~\cite{badia, ba2, olk,  go, gore2, gr, ko,  pi, cro}) and results on all sorts
of proof-theoretic and model-theoretic properties of these systems have been
produced. Historically, before Rauszer, at least one place where
bi-intuitionistic logic was inadvertently introduced and studied is~\cite{klem},
but since the article was in German and the main topic was not bi-intuitionistic
logic, it did not attract much attention from the community of researchers in
the area until recently.

The present contribution is focused on completing the picture of what exactly can be expressed model-theoretically in the language of propositional bi-intuitionistic logic. In~\cite{badia}, directed bisimulations were isolated as the correct model-theoretic relations to distinguish the standard translations of bi-intuitionistic formulas into first-order logic from other first-order formulas  by means of a van Benthem characterization theorem. Following this work, in~\cite{olk} these relations, in conjunction with other natural model-theoretic properties, were used to characterize the expressive power of bi-intuitionistic logic among all its possible expressive extensions through a Lindstr\"om-style theorem. 
Two interesting results  that serve as inspiration for the work  in the current paper are the main theorems from~\cite{gold} and~\cite{gold2}. On the one hand, \cite{gold} contains the so called Goldblatt-Thomason theorem for modal logic which provides a model-theoretic characterization of the classes of frames that can be axiomatised by a modal theory among the classes that are already axiomatisable in first-order logic.
On the other hand, in~\cite{gold2}, Goldblatt provides a model-theoretic characterization of the classes of intuitionistic Kripke models axiomatisable in propositional intuitionistic logic. Our goal in the present work is to produce analogues of the  theorems in~\cite{gold} and~\cite{gold2} for bi-intuitionistic logic.

The article is arranged as follows. In Section~\ref{sec:prelim}
, we give the basic definitions of the syntax and semantics of bi-intuitionistic logic, some algebraic background related to double-Heyting algebras, and the standard translation of bi-intuitionistic logic into first-order logic. In Section~\ref{sec:mod-and-frm}, we introduce the model-theoretic constructions and relations between structures that we will use in our main theorems. These include notions of countably saturated structures, bounded morphisms, subframes (and submodels), disjoint unions, directed bisimulations and prime filter extensions. Finally, we also provide some observations on algebraic duality between some of these relations and corresponding ones between algebras.
In Section~\ref{sec:axiomatisable-models}  we provide a model-theoretic characterization using some of the notions previously introduced of which classes of models exactly correspond to models of a set of bi-intuitionistic formulas (Theorem~\ref{axmod}). In Section~\ref{sec:Goldblatt-Thomason} we provide a direct analogue of the famous Goldblatt-Thomason theorem and in Section~\ref{con} we give a brief summary of what we have accomplished.

\section{Preliminaries}\label{sec:prelim}
We begin by defining the language of bi-intuitionistic logic and its standard
Kripke semantics. 
 Let $\Prop$ be an arbitrary but fixed set of proposition letters.
  Then the language $\biLan(\Prop)$ is the set of formulas generated by the
  grammar
  $$
    \phi ::= p
      \mid \top
      \mid \bot
      \mid \phi \wedge \phi
      \mid \phi \vee \phi
      \mid \phi \to \phi
      \mid \phi \bito \phi,
  $$
where $p$ ranges over $\Prop$. 
We will write $\biLan$ instead of
$\biLan(\Prop)$, unless the set of propositional letters changes to something
other than $\Prop$, as for example, in Lemma~\ref{lem:satur}.

\begin{definition}\label{def:frame}
A \emph{Kripke frame} is any poset $(W, \leq)$.
Elements of\/ $W$ are often called \emph{worlds}. 
\end{definition}

Let $(W, \leq)$ be a Kripke frame and let $a \subseteq W$.
We write $\uparr a$ for the upward closure of $a$ and
$\downarr a$ for the downward closure of $a$.
A subset equal to its upward closure is an \emph{upset}, and
a subset equal to its downward closure is a \emph{downset}.

A \emph{valuation} for a Kripke frame $(W, \leq)$ is a map $V$
  that assigns to each proposition letter $p \in \Prop$ an upset
  $V(p)$ of $(W, \leq)$.
  A \emph{Kripke model} is a tuple $\mo{M} = (W, \leq, V)$ consisting of a
  Kripke frame $(W, \leq)$ and a valuation $V$.
  We say that the Kripke model $\mo{M}$ is \emph{based on}
  the Kripke frame $(W, \leq)$, and that $(W, \leq)$ is the
  \emph{underlying Kripke frame} of $\mo{M}$.
We will often use $\mathfrak{F}$ to stand for an arbitrary Kripke frame, and write
$\mo{M} = (\mathfrak{F}, V)$ for a model based on $\mathfrak{F}$.

\begin{definition}\label{def:truth-in-model} 
Truth of a formula $\phi$ at a world $w$ of a model $\mo{M} = (\mathfrak{F},V)$ based on a
frame $\mathfrak{F} = (W,\leq)$ is denoted symbolically by\quad $\mo{M}, w \Vdash \phi$\quad and  
defined recursively as follows:
  \begin{align*}
    \mo{M}, w \Vdash p &\iff x \in V(p) \\
    \mo{M}, w \Vdash \top &\phantom{\iff} \text{always} \\
    \mo{M}, w \Vdash \bot &\phantom{\iff} \text{never} \\
    \mo{M}, w \Vdash \phi \wedge \psi
      &\iff \mo{M}, w \Vdash \phi \text{ and } \mo{M}, w \Vdash \psi \\
    \mo{M}, w \Vdash \phi \vee \psi
      &\iff \mo{M}, w \Vdash \phi \text{ or } \mo{M}, w \Vdash \psi \\
    \mo{M}, w \Vdash \phi \to \psi
      &\iff \forall v \in W (\text{if } w \leq v 
                           \text{ and } \mo{M}, v \Vdash \phi
                          \text{ then } \mo{M}, v \Vdash \psi) \\
    \mo{M}, w \Vdash \phi \bito \psi
      &\iff \exists v \in W (v \leq w 
                           \text{ and } \mo{M}, v \Vdash \phi
                           \text{ and not } \mo{M}, v \Vdash \psi)
  \end{align*}
  For a set\/ $\Gamma$ of\/ $\biLan$-formulas we write
  $\mo{M}, w \Vdash \Gamma$ if\/  $\mo{M}, w \Vdash \phi$ for all $\phi \in \Gamma$.
\end{definition}

As usual, we extend the function $V$ from $\Prop$ to the collection of
  all formulas by defining 
  $$
    V(\phi) := \{ w \in W \mid \mo{M}, w \Vdash \phi \}.
    $$
This set is often called the \emph{truth set} of $\phi$.     
If $V(\phi) = W$, then we write $\mo{M} \Vdash \phi$.

The notion of truth at a world naturally extends to frames. For a frame
$\mathfrak{F} = (W,\leq)$ and $w \in W$, we write  
  $\mathfrak{F}, w \Vdash \phi$ if $\mo{M}, w \Vdash \phi$ for every $\mo{M}$ based on $(W,
  \leq)$. Further still, if $\mathfrak{F}, w \Vdash \phi$ holds for every $w \in W$, or,
  equivalently, if 
  $\mo{M} \Vdash \phi$ for  every $\mo{M}$ based on $(W, \leq)$, we say that
  $\phi$ is \emph{valid} on $\mathfrak{F}$, and write $\mathfrak{F} \Vdash \phi$ relying on context to
  distinguish between models and frames.   

If $\phi$ is valid on every frame $\mathfrak{F}$ from some collection $K$ of
frames, we say that $\phi$ is \emph{valid in $K$} and write
$K\Vdash \phi$.

\subsection{Upset algebras and double-Heyting algebras}\label{ssec:dbl-HA}

For any frame $\mathfrak{F}$ and any valuation $V$, the truth sets naturally form an
algebra of upsets. To begin with, we have
$V(\phi \wedge \psi) = V(\phi) \cap V(\psi)$, $V(\phi \vee \psi) = V(\phi)
\cup V(\psi)$ and $V(\top) = W$, $V(\bot) = \emptyset$, so the truth sets
ordered by inclusion form 
a bounded distributive lattice of upsets of $\mathfrak{F}$. In general this lattice is a
sublattice of the (bounded, distributive) lattice of all upsets of $\mathfrak{F}$.
  Furthermore, if we unfold the definitions of $\to$ and $\bito$ we see that
  the truth sets of $\phi \to \psi$ and $\phi \bito \psi$
  in a Kripke model $\mo{M} = (W, \leq, V)$
  can be defined as follows:
  \begin{align*}
    V(\phi \to \psi)
      &= W \setminus \downarr (V(\phi) \setminus V(\psi)) \\
    V(\phi \bito \psi)
      &= \uparr(V(\phi) \setminus V(\psi))
  \end{align*}
  Importantly, these sets are both upsets of $\mathfrak{F}$. The algebra defined
  this way will be called the \emph{algebra of definable upsets} in
  Section~\ref{ssec:def-p-filt-ext}.  Applying exactly the same
definitions to arbitrary upsets of $\mathfrak{F}$ (not necessarily values of any
valuation), we obtain the \emph{upset algebra} of $\mathfrak{F}$, denoted $\mathrm{Up}(\mathfrak{F})$.
It is easy to verify that $\mathrm{Up}(\mathfrak{F})$ satisfies the following equivalences:
\begin{align*}
a\cap b \subseteq c &\iff b\subseteq a\to c,\\
a\cup b \supseteq c &\iff b \supseteq c\bito a.
\end{align*}
It turns out that these are precisely the definitional properties of
\emph{double-Heyting algebras}.

\begin{definition}\label{def:dbl-H}
An algebra  $\mathbf{A} = (A;\wedge,\vee,\to,\bito,0,1)$, such that
$(A;\wedge,\vee,0,1)$ is a bounded distributive lattice and the equivalences  
\begin{align*}
x\wedge y \leq z &\iff y\leq x\to z,\\
x\vee y \geq z &\iff y \geq z\bito x,
\end{align*}
hold for all $x,y,z\in A$, is called a double-Heyting algebra. 
\end{definition}  

These properties are in fact equivalent to equations, so the class
$\mathsf{DH}$ of all double-Heyting algebras is an \emph{equational class},
hence, by the easy half of Birkhoff Theorem (see, e.g.,~\cite{Ber11} Lem.~4.36
and Thm.~4.41), a \emph{variety}, that is, a class closed 
under homomorphic images, subalgebras and direct products.
The connection between double-Heyting
algebras and bi-intuitionistic logic is very tight, in fact
$\biInt$ is \emph{algebraizable with equivalent algebraic semantics}
$\mathsf{DH}$.

Details of algebraizability are beyond the scope of this article, but to keep it
as self-contained as sensible we give a sketch. Let $\Sigma$ be a purely
functional signature, say, $\Sigma = \{\wedge,\vee,\to,\bito\}$.
Let $\mathbf{L}$ be a logic in the language of $\Sigma$, 
given as a deductive system $(\Gamma_\Sigma,\vdash)$ where
$\Gamma_\Sigma$ is a set of formulas over $\Sigma$
and $\vdash$ a Tarskian consequence
relation. Let $\mathcal{K}$ be a class of algebras of $\Sigma$, and let
$E_\Sigma$ be the set of all equations over $\Sigma$ (formally
$E_\Sigma = \Gamma_\Sigma\times \Gamma_\Sigma$).
We say that $\mathbf{L}$ is algebraizable with equivalent algebraic
  semantics $\mathcal{K}$, if there are
two translation maps $\varepsilon\colon \Gamma_\Sigma\to E_\Sigma$ and
$\delta\colon E_\Sigma\to \Gamma_\Sigma$, which behave as the diagrams below
indicate. 
$$
\delta(\varepsilon(\Gamma))\dashv\vdash\Gamma
\qquad
\begin{tikzcd}
\Gamma\vdash\Pi \ar[r, leftrightarrow, "\text{iff}"] &
\varepsilon(\Gamma)\models\varepsilon(\Pi) \ar[d,hook,"\text{case of}"]\\
  \delta(\Psi)\vdash \delta(\Phi) \ar[u,hook, "\text{case of}"]
  \ar[r, leftrightarrow, "\text{iff}"] & \Psi\models\Phi\\
\end{tikzcd}
\qquad
\varepsilon(\delta(\Psi))\sledom\kern-1.5pt\models\Psi
$$
Here, $\models$ is the usual equational (first-order) consequence relation,
$\Gamma\vdash\Pi$ means $\Gamma\vdash\pi$ for all $\pi\in \Pi$, 
similarly for $\Psi\models\Phi$, and
$\dashv\vdash$, $\sledom\kern-1.5pt\models$ mean that the relevant deducibility
holds both ways.

\begin{proposition}\label{prop:algebraizability}
The logic $\biInt$ is algebraizable with equivalent algebraic semantics
$\mathsf{DH}$. The translations are given by:
$$
\varphi\overset{\varepsilon}{\mapsto} (\varphi = 1)\qquad
\text{and}\qquad
(\tau_1 = \tau_2)\overset{\delta}{\mapsto}
(\tau_1\to \tau_2)\wedge(\tau_2\to\tau_1).
$$
Subvarieties
of\/ $\mathsf{DH}$ correspond in a similar manner to axiomatic extensions of\/
$\biInt$.
\end{proposition}

For more on double-Heyting algebras, including a very nice crash course on
the duality theory for them, see~\cite{Tay17}.  

\subsection{Standard translation}\label{ssec:ST}

Consider a first order language with a binary relation symbol $R$,
and a unary predicate $P$ for each $p \in \Prop$. Following the tradition in
modal logic~\cite{blackburn}, we call it the \emph{correspondence language}  $\biLan^{corr}$
for  $\biLan$. Note that each model $\mo{M}$ for $\biLan$ is also
a model for $\biLan^{corr}$ in the usual first-order sense, namely,
$W$ is the domain, $\leq$ is the interpretation of $R$, and $V(p)$ is the
interpretation of $P$. 
The usual first-order satisfaction relation will be denoted by $\vDash$, and by
$\mo{M} \vDash \phi [a]$ we mean that $a\in W$ satisfies the first-order formula
$\phi$.  

The \emph{standard translation} of bi-intuitionistic formulas to
$\biLan^{corr}$ is defined inductively: 
\begin{align*}
ST_x(\bot) &= \bot\\
ST_x(\top) &=  \top\\
ST_x(p) &=   Px\\
ST_x(\phi \wedge \psi) &= ST_x(\phi) \wedge ST_x(\psi)\\         
ST_x(\phi \vee \psi) &= ST_x(\phi) \vee ST_x(\psi)\\
ST_x(\phi\rightarrow\psi) &= \forall y(\neg (Rxy \wedge ST_y(\phi))\vee ST_y(\psi))\\
ST_x(\phi \bito \psi) &= \exists y (Ryx \wedge ST_y(\phi)\wedge \neg ST_y(\psi)),  
\end{align*}
with the usual proviso of choosing $y$ carefully to avoid free variable capture.

\begin{proposition} \label{pro:tr}  For any Kripke model
$\mo{M} = (W,\leq, V)$, any  $w\in W$, and any bi-intuitionistic formula
$\phi$, we have:
$$
\mo{M}, w \Vdash \phi\qquad \text{if and only if}\qquad \mo{M} \vDash ST_x(\phi) [w].
$$
\end{proposition}

\begin{proof} We omit the proof, which would proceed by straightforward induction on
the complexity of formulas.
\end{proof} 

\section{Some model and frame constructions}\label{sec:mod-and-frm}

The results below will involve some machinery from the classical model theory of
first-order logic and thus we will start the section by recalling a few useful
notions which can be found in any of the standard references,
such as~\cite{hodges}.  

\subsection{Countably saturated structures}\label{ssec:satur}
We will use $\overline{a}$, $\overline{b}$, and $\overline{x}$, $\overline{y}$
to denote sequences of elements of a model and variables, respectively. 
For a model $\mo{M}$, by $\dom{\mo{M}}$ we denote the domain of $\mo{M}$.
In particular, if
$\mo{M}$ is a model based on a Kripke frame $(W,\leq)$, then $\dom{\mo{M}} = W$. 
If $X \subseteq \dom{\mo{M}}$, then $(\mo{M}, a)_{a \in X}$ is the \emph{expansion}
of $\mo{M}$ obtained by adding a constant $c_a$ for each $a \in X$.

We say that a set of first order formulas $\Phi(\overline{x})$  is
\emph{realisable} in a model $\mo{M}$ if there is some sequence $\overline{a}$
of elements of $W$ such that $\mo{M} \vDash \Phi [\overline{a}]$.
$\Phi(\overline{x})$ is said to be \emph{refutable} in $\mo{M}$  if
$\Phi^{\prime}(\overline{x}) = \{ \neg \phi : \phi \in \Phi(\overline{x})\}$ is
realisable in $\mo{M}$.

\begin{definition}
  Let $\lambda$ be a cardinal. A model $\mo{M}$ of  first order
  logic is said to be \emph{$\lambda$-saturated} if whenever $X \subseteq
  \dom{\mo{M}}$ and $|X| \ < \ \lambda$, then  some expansion $(\mo{M}, a)_{a
    \in X}$ of $\mo{M}$  realises every set of formulas $\Phi(x)$ of the
  language of $(\mo{M}, a)_{a \in X}$ which is consistent with the set of all
  first order sentences true in $(\mo{M}, a)_{a \in X}$.
\end{definition} 

 Note that if $\kappa \ < \ \lambda$, then $\lambda$-saturation implies
 $\kappa$-saturation. 

\begin{definition} Let $\mo{M}$ and $\mo{N}$ be two first order models.
$\mo{N}$ is an \emph{elementary extension} of $\mo{M}$ if $\mo{M}$ is
(isomorphic to) a submodel $\mo{N}^{\prime}$ of $\mo{N}$ such that where
$\overline{a}$ is a sequence of elements of $\mo{N}^{\prime}$, for each first
order formula $\phi$, 
$$
\mo{N}^{\prime} \vDash \phi[\overline{a}] \iff  \mo{N} \vDash \phi[\overline{a}].
$$
\end{definition}

The following proposition is a corollary of Theorems~6.1.4 and~6.1.8 in~\cite{CK73}.
The method of obtaining such saturated models is often
referred to as the Keisler method.

\begin{proposition}\label{prop:om-satur-ext}
Each model $\mo{M}$ for\/ $\biLan^{corr}$
has an $\omega$-saturated elementary extension, obtainable as an ultrapower of
$\mo{M}$.
\end{proposition}

\subsection{Bounded morphisms, generated subframes, disjoint
  unions}\label{ssec:b-morph-gen-subfr-disj-un} 

Bounded morphisms between Kripke frames and models are the maps that
have the correct functorial properties in category theory sense. For models,
in particular, they preserve and reflect truth of $\biLan$-formulas.  
They map valuations to valuations, and
if $w$ is a world and $f$ is a bounded morphism,
then every formula true at $w$ is also true at $f(w)$ (preservation),
and every formula true at $f(w)$ is also true at $w$ as well (reflection).

\begin{definition}\label{mt:def:bd-mor}
Let $\mo{F} = (W, \leq)$ and $\mo{F}' = (W', \leq')$ be two Kripke frames.
A function $f : W \to W'$ is a \emph{bounded morphism}
from $\mo{F}$ to $\mo{F}'$ if it is order-preserving and moreover,
for all $w, v \in W$ and $u' \in W'$:
\begin{enumerate}
\item[\textup{(}ub\textup{)}] \label{mt:it:bm-b1}
         If $f(w) \leq' u'$ then there exists a $u \in W$
          such that $w \leq u$ and $f(u) = u'$;
\item[\textup{(}lb\textup{)}] \label{mt:it:bm-b2}
          If $u' \leq' f(w)$ then there exists a $u \in W$
          such that $u \leq w$ and $f(u) = u'$.
        \end{enumerate}
If $f$ is surjective, we say that $\mo{F}'$ is a \emph{bounded-morphic image} of
$\mo{F}$.          
\end{definition}

  
\begin{definition}\label{mt:def:bd-mor-model}
  Let $\mo{M} = (\mo{F}, V)$ and $\mo{M}' = (\mo{F}', V')$ be two Kripke models.
  A function $f : W \to W'$ is a \emph{bounded morphism}
  from $\mo{M}$ to $\mo{M}'$ if it is a bounded morphism
  from $\mo{F} \to \mo{F}'$ and for all $w \in W$ and $p \in \Prop$:
\begin{enumerate}
\item[\textup{(}prop\textup{)}] \label{mt:it:bm-prop}
          $w \in V(p)$ if and only if $f(w) \in V'(p)$.
\end{enumerate}
\end{definition}

\begin{proposition}\label{mt:prop:mor-pres-truth}
  Let $\mo{M} = (\mo{F}, V)$ and $\mo{M}' = (\mo{F}', V')$ be two Kripke
  models and $f : \mo{M} \to \mo{M}'$ a bounded morphism.
  Then for all $w \in W$ and $\phi \in \biLan$ we have
  $$
    \mo{M}, w \Vdash \phi
      \iff \mo{M}', f(w) \Vdash \phi.
  $$
\end{proposition}
\begin{proof}
Induction on the complexity of $\phi$. We leave it to the reader as an exercise.
\end{proof}

\begin{definition}\label{def:gen-subfr}
Let $\mo{M}$ and $\mo{N}$ be Kripke models such that  
$\mo{M}$ is a substructure of $\mo{N}$ is the first-order sense. If the
natural inclusion map is a bounded morphism, then we say that
$\mo{M}$ is a \emph{generated submodel} of $\mo{N}$. The same, \emph{mutatis mutandis},
applies to Kripke frames, the resulting notion being that of a \emph{generated
  subframe}.  
\end{definition}  

\begin{lemma}\label{lem:gen-subfr}
Let\/ $\mo{F} = (W,\leq)$ and $\mo{F}' = (W',\leq')$ be Kripke frames. Then,
$\mo{F}'$ is a generated subframe of\/ $\mo{F}$ if and only if  (i) $W'\subseteq
W$, (ii) $\leq'$ is the restriction of $\leq$ to $W'$, and (iii) $W'$ is both
upward and downward closed in $W$.

If\/ $\mo{M} = (\mo{F}, V)$ and\/ $\mo{M}' =
(\mo{F}',V')$ are Kripke models, then $\mo{M}'$ is a generated submodel of\/
$\mo{M}$ if and only if (i)--(iii) hold and moreover (iv) $V'$ is the
restriction of $V$ to $W'$.
\end{lemma} 
  
\begin{proof}
Straightforward from the definitions.
\end{proof}

Yet another frame construction we will need is the disjoint union of
frames and of models.

\begin{definition}\label{def:disj-union}
Let $\mathcal{F} = \{\mo{F}_i: i\in I\}$ be a set of Kripke frames
with $\mo{F}_i = (W_i,\leq_i)$, and let
$\mathcal{M} = \{(\mo{F}_i, V_i): i\in I\}$ be a set of Kripke models.  
The \emph{disjoint union} of $\mathcal{F}$ is defined by
$$
    \coprod_{i \in I} \mo{F}_i = (\biguplus W_i, \biguplus \leq_i),
$$
where $\biguplus W_i$ and $\biguplus \leq_i$ are the disjoint union of the sets
$W_i$ and the disjoint union of the relations $\leq_i$, respectively. 
Similarly, the disjoint union of $\mathcal{M}$ is defined  by
$$
    \coprod_{i \in I} (\mo{F}_i, V_i) = (\coprod_{i \in I} \mo{F}_i, V),
$$
where for $w \in W_i$, $w \in V(p)$ if $w \in V_i(p)$.
\end{definition}

Note that for sets $\mathcal{F}$ and $\mathcal{M}$ above,
the identity function $\iota_j(w) = w$, for each $w \in W_j$, is a bounded
morphism between Kripke frames and models, respectively.
The proposition below follows easily from the definitions involved.

\begin{proposition}\label{prop:truth-pres}
Let $\{(\mo{F}_i, V_i): i \in I \}$ be a set of Kripke models
based on frames $\mo{F}_i = (W_i, \leq_i)$. Let
$(\mo{A}, V^A)$ be a Kripke model based on a frame $\mo{A} = (A,\leq^A)$
and let $(\mo{B}, V^B)$ be a Kripke model based on a frame $\mo{B} = (B,\leq^B)$.
Then, for any $\phi \in \biLan$, the following hold:
\begin{enumerate}
\item $(\mo{F}_i, V_i) \Vdash \phi$ for all $i \in I$, if and only if\/
          $\coprod_{i \in I} (\mo{F}_i, V_i) \Vdash \phi$.
\item $\mo{F}_i \Vdash \phi$ for all $i \in I$, if and only if\/
          $\coprod_{i \in I} \mo{F}_i\Vdash \phi$.
\item If\/ $(\mo{A},V^A)$ is a generated submodel of $(\mo{B},V^B)$
and $\mo{B} \Vdash \phi$, then $\mo{A} \Vdash \phi$.
\item If\/ $\mo{A}$ is a generated subframe of $\mo{B}$ and
          $\mo{B} \Vdash \phi$, then $\mo{A}\Vdash \phi$.
\item If\/ $(\mo{A},V^A)$ is a bounded-morphic image of $(\mo{B},V^B)$ and
          $(\mo{B},V^B) \Vdash \phi$ then $(\mo{A},V^A) \Vdash \phi$.
\item If\/ $\mo{A}$ is a bounded-morphic image of $\mo{B}$ and
          $\mo{B} \Vdash \phi$ then $\mo{A} \Vdash \phi$.
\end{enumerate}
\end{proposition}

\subsection{Shadows of duality}\label{ssec:duality}

Duality theory for double-Heyting algebras piggybacks on Priestley duality, of which
we assume the reader to have some knowledge. However, since the whole
topological machinery will not be used, we will recall the concepts and results
absolutely necessary to follow our arguments. For a very readable
account of Priestley duality we refer the reader to~\cite{ilo}.
For the specific case of double-Heyting algebras~Ch. 2 of~\cite{Tay17} serves as
a good reference. 

We mentioned upset algebras in Section~\ref{ssec:dbl-HA} briefly, now we will
give proper definitions.

\begin{definition}\label{def:upset-alg}
Let $\mo{F} = (W, \leq)$ be a Kripke frame. The \emph{upset algebra}
$\mo{F}^+$ of $\mo{F}$ is the structure
$\left\langle \mathrm{Up}(W); \cap, \cup, \rightarrow, \bito, 0, 1\right\rangle$,
where $\mathrm{Up}(W)$ is the set of all upsets of $(W, \leq)$, the operations
$\cap$ and $\cup$ are set-theoretical,  $0 = \emptyset$, $1 = W$,
and the remaining operations are defined by:
\begin{itemize}
\item[] $X \bito Y = \{w \in W : \exists v \in W (v\leq w \mathbin{\&} v \in X
\mathbin{\&} v \notin Y)\}$,
\item[] $X \rightarrow Y = \{w \in W: \forall v \in W ((w\leq v \mathbin{\&} v
  \in X) \Rightarrow v \in Y)\}$. 
\end{itemize}
\end{definition}

The reader is invited to check that the definitions of $\to$ and $\bito$ given
above coincide with these from Section~\ref{ssec:dbl-HA}. The algebraic
structure of $\mo{F}^+$ does not depend on valuations for $\mo{F}$, and
therefore for any Kripke \emph{model} $\mo{M}$ we will write  
$\mo{M}^+$ for the upset algebra of the underlying frame. 

The next proposition is not difficult to verify using observations from
Section~\ref{ssec:dbl-HA}. We leave it as an exercise
to the reader.

\begin{proposition} Let $\mo{M}$ be a Kripke model for bi-intuitionistic
logic. Then $\mo{M}^+$ is a double-Heyting algebra.
\end{proposition}

Note that for a specific model $\mo{M}$ the algebra $\mo{M}^+$ can have too
many elements, as not every upset is guaranteed to be a value of a
formula. However, for a frame $\mo{F}$, the upset algebra is exactly what we
need, each upset is a value of a formula in \emph{some} valuation.   

Having thus constructed algebras from frames, we want to go back and construct
frames from algebras. As usual, this is trickier. First, let us recall
the notion of a \emph{prime filter} in a lattice $L$. Namely, a subset $F$
of $L$ is a \emph{filter} if it is a nonempty upset, and moreover
$x,y\in F$ implies $x\wedge y\in F$. A filter $F$ is \emph{prime}
if $x\vee y\in F$ implies $x\in F$ or $y\in F$. 

\begin{definition}\label{def:p-filter-frame}
Let $\mathbf{A}$ be a double-Heyting algebra, and let
$\mathcal{F}_p(\mathbf{A})$ be the set of all prime filters of the underlying
lattice of $\mathbf{A}$. The \emph{prime filter frame} $\mathbf{A}_+$ of
$\mathbf{A}$ is the poset $(\mathcal{F}_p(\mathbf{A}),\subseteq)$.
\end{definition}

Naturally, for any frame $\mo{F}$ we can form the frame $(\mo{F}^+)_+$, and for
any algebra $\mathbf{A}$ we can form the algebra $(\mathbf{A}_+)^+$. For finite
frames and finite algebras, we have a perfect situation:

\begin{proposition}\label{prop:fin-duality}
Let $\mo{F}$ be a finite frame, and let $\mathbf{A}$ be a finite double-Heyting
algebra. Then, $(\mo{F}^+)_+\cong \mo{F}$ and
$(\mathbf{A}_+)^+ \cong \mathbf{A}$.
\end{proposition}  

For infinite frames and algebras isomorphism no longer holds: algebras have too
many prime filters, and frames have too many upsets. Topology comes to the
rescue, but as we mentioned already, we will not need it.
We will, however, need
the concept of a double-dual model, in Section~\ref{ssec:def-p-filt-ext} and later. 

\begin{definition}
For a model $\mo{M} = (\mo{F}, V)$, with $\mo{F} = (W, \leq)$, the
\emph{double-dual model} $(\mo{M}^+)_+ = \bigl((\mo{F}^+)_+, V^*\bigr)$ of\/
$\mo{M}$, is defined by setting $V^*(p) = \{ w \in \mathcal{F}_p(\mo{F}^+): V(p)
\in w\}$. 
\end{definition}

\begin{lemma}\label{lem:modus-ponens}
Let $\mo{F} = (W, \leq)$. Then, in $(\mo{F}^+)_+$ we have
$$
w\subseteq v \iff \forall x, y \in \mathrm{Up}(W)
(x \rightarrow y \in w \mathbin{\&} x \in v \ \Rightarrow\  y \in v).
$$
\end{lemma}  

\begin{proof}
First, observe that the following holds for any $x,y\in\mathrm{Up}(W)$:
\begin{equation}\tag{\dag}\label{eq:dag}
\bigl(W\setminus \downarr(x\setminus y)\bigr) \cap x \subseteq y.
\end{equation}
To see it, take
any $a\notin \downarr(x\setminus y)$ such that $a\in x$.
Then $\uparr{a}\cap (x\setminus y) =\emptyset$. Since $a\in x$, this implies
$\uparr{a}\cap (W\setminus y) = \emptyset$, that is,
$\uparr{a}\subseteq y$, hence $a\in y$.

Now, assume $w\subseteq v$ and take $x,y\in\mathrm{Up}(W)$ such that
$x \rightarrow y \in w$ and $x \in v$. By definition of $\rightarrow$,
we have $W\setminus \downarr(x\setminus y)\in w$. So, by assumption
$W\setminus \downarr(x\setminus y)\in v$. Since $v$ is a filter,
$\bigl(W\setminus \downarr(x\setminus y)\bigr) \cap x\in v$ so by~\eqref{eq:dag}
$y\in v$.

Conversely, assume that for all $x, y \in \mathrm{Up}(W)$ such that
$x \rightarrow y \in w$ and $x \in v$ we have $y\in v$.
Take any $z\in w$. Since $w$ is a filter, for an arbitrary $x$ we have
$x\rightarrow z\in w$.  Pick any $x\in v$. Then, by assumption
$z\in v$, proving $w\subseteq v$.
\end{proof}

We end this section by two propositions stating the results we will need in
Section~\ref{sec:Goldblatt-Thomason}.

\begin{proposition}\label{prop:min-dual-frm}
Let $\{\mo{F}_i: i \in I \}\cup\{\mo{G},\mo{H}\}$ be a set of Kripke frames. 
The following hold:  
\begin{enumerate}
\item $\left(\coprod_{i \in I} \mo{F}_i\right)^+ \cong \prod_{i \in I} \mo{F}^+_i$.
\item If\/ $\mo{G}$ is a generated subframe of\/ $\mo{H}$, then
  $\mo{H}^+$ is a homomorphic image of\/ $\mo{G}^+$.
\item If\/ $\mo{G}$ is a bounded-morphic image of\/ $\mo{H}$, then
  $\mo{H}^+$ is a subalgebra of\/ $\mo{G}^+$.  
\end{enumerate}
\end{proposition}

\begin{proposition}\label{prop:min-dual-alg}
Let $\mathbf{A}$ and $\mathbf{B}$ be 
double-Heyting algebras. The following hold:
\begin{enumerate}
\item $\mathbf{A}$ is a subalgebra of $(\mathbf{A}_+)^+$.
\item If\/ $\mathbf{A}$ is a homomorphic image of \/ $\mathbf{B}$, then
  $\mathbf{B}_+$ is a generated subframe of\/ $\mathbf{A}_+$.
\item If\/ $\mathbf{A}$ is a subalgebra of\/ $\mathbf{B}$, then
  $\mathbf{B}_+$ is a bounded-morphic image of\/ $\mathbf{A}_+$.  
\end{enumerate}
\end{proposition}  

The double dual algebra $(\mathbf{A}_+)^+$  is known as the
\emph{canonical extension} of $\mathbf{A}$.

\subsection{Directed bisimulations}\label{ssec:dir-bisim}

The relations (directed bisimulations) that we introduce in this section originally appeared in~\cite{kurto} in a different form. The current version has essentially been used in~\cite{badia, olk}.

\begin{definition}
  Let $\mo{M} = (W, \leq, V)$ and $\mo{M}' = (W', \leq', V')$ be two
  Kripke models. A \emph{directed bisimulation} between $\mo{M}$ and $\mo{M}'$
  is a pair $(Z_1, Z_2)$ of relations $Z_1 \subseteq W \times W'$ and
  $Z_2 \subseteq W' \times W$ such that:
  \begin{enumerate}
  \renewcommand{\theenumi}{D$_{\arabic{enumi}}$}
  \renewcommand{\labelenumi}{\textup{(}\theenumi\textup{)} }
    \item \label{mt:it:dbisim-prop}
          if $wZ_1w'$ and $w \in V(p)$, then $w' \in V'(p)$, for all $p \in \Prop$;
    \item \label{mt:it:dbisim-prop2}
          if $w'Z_2w$ and $w' \in V'(p)$, then $w \in V(p)$, for all $p \in \Prop$;
    \item \label{mt:it:dbisim-f1}
          if $wZ_1w'$ and $w' \leq' v'$, then $\exists v \in W$ such that
          $w \leq v$ and $vZ_1v'$ and $v'Z_2v$; 
    \item \label{mt:it:dbisim-b1}
          if $w'Z_2w$ and $w \leq v$, then $\exists v' \in W'$ such that
          $w' \leq' v'$ and $v'Z_2v$ and $vZ_1v'$;
    \item \label{mt:it:dbisim-f2}
          if $wZ_1w'$ and $v \leq w$, then $\exists v' \in W'$ such that
          $v' \leq' w'$ and $vZ_1v'$ and $v'Z_2v$; 
    \item \label{mt:it:dbisim-b2}
          if $w'Z_2w$ and $v' \leq' w'$, then $\exists v \in W$ such that
          $v \leq w$ and $v'Z_2v$ and $vZ_1v'$.
  \end{enumerate}
  Two worlds $w \in W$ and $w' \in W'$ are called \emph{directed bisimilar}
  if there exists a directed bisimulation $(Z_1, Z_2)$ from $\mo{M}$
  to $\mo{M}'$ such that $wZ_1w'$.
  We denote this by $\mo{M}, w \dbisim \mo{M}', w'$.
\end{definition}

Directed bisimilar states preserve truth,
as shown in the following proposition. (For a proof see~\cite{badia}.)

\begin{proposition}
  Let $\mo{M} = (W, \leq, V)$ and $\mo{M}' = (W', \leq', V')$ be
  Kripke models, $w \in W$ and $w' \in W'$.
 If $\mo{M}, w \dbisim \mo{M}', w'$, then  for any formula  $\phi \in \biLan$ we have that $      \mo{M}, w \Vdash \phi  $ only if\/ $\mo{M}', w' \Vdash \phi $.
\end{proposition}


Bounded morphisms give rise to directed bisimulations.

\begin{proposition}\label{bmdr}
Let $\mo{M}_1 = (W_1, \leq_1, V_1)$ and $\mo{M}_2 = (W_2, \leq_2, V_2)$ be two
Kripke models for bi-intuitionistic logic. Furthermore, let $f: W_1
\longrightarrow W_2$ be a bounded morphism. Then, the pair $\langle Z_1,
Z_2\rangle$ is a  directed bisimulation, where: 
\begin{align*}
xZ_1y &\iff f(x) \leq_2 y,\\
xZ_2y &\iff x \leq_2 f(y).
\end{align*}
Moreover, if  $f$ is surjective,  then the  directed bisimulation  $\langle Z_1,
Z_2\rangle$ is surjective with respect to both   $\mo{M}_1$ and  $\mo{M}_2$.
\end{proposition}

\begin{proof} 
Let  $i, j \in \{1, 2\}$, $i \neq j$. First, if $xZ_i y$, we immediately have
$\mo{M}_i, x \Vdash p$ only if $\mo{M}_j, y \Vdash p$ by persistence of atomic
formulas across the partial ordering $\leq_i$. 

If $xZ_1y$ (i.e. $f(x) \leq_2 y$) and $y\leq_2b$ for some $b\in W_2$, then
$f(x) \leq_2 b$ and we have that there exists $b^{\prime}\in W_1$ such that
$x\leq_1 b^{\prime}$ and $f(b^{\prime}) = b$  by properties of bounded
morphisms. Thus $bZ_2 b^{\prime}$ since  $b \leq_2 f(b^{\prime}) $  and
$b^{\prime}Z_1b$ as  $f(b^{\prime}) \leq_2 b$.

If $xZ_1y$ (i.e. $f(x) \leq_2 y$) and $b\leq_1x$ for some $b\in W_1$, so
$f(b)\leq_2f(x)$ by properties of bounded morphisms, and thus $f(b) \leq_2
y$. But then  $bZ_1 f(b)$ since $f(b)\leq_1 f(b)$ and $f(b)Z_2b$ given that
$f(b)\leq_1 f(b)$.  
 
The remaining conditions are shown similarly. Finally, suppose that  $f$ is
surjective. Then if $y\in W_1$, we have that $f(y)Z_2y$ since $f(y) \leq_2
f(y)$. Moreover, if $y\in W_2$ there must be $x$ such that $f(x)=y$ by
surjectivity, so  $f(x) \leq_2 y$ and hence $xZ_1y$ as desired. 
\end{proof}

\begin{proposition}\label{rdbp}  Let $\mo{M} = (W, \leq, V)$ and $\mo{M}' = (W',
  \leq', V')$ be two 
  Kripke models and  $\langle Z_1, Z_2 \rangle$
  \textup{(}where $Z_1\subseteq W\times
  W'$ and $Z_2\subseteq W'\times W$\textup{)} be a  directed bisimulation surjective
  with respect to $\mo{M}'$. Then $\mo{M} \Vdash \phi$ only if\/ $\mo{M}' \Vdash \phi$,
  for every formula $\phi$ of\/ $\biLan$.
\end{proposition} 
  
\begin{proof} This is a corollary of the preservation under  directed
  bisimulations of the formulas of $\biLan$. Suppose that $\mo{M}' \not \Vdash
  \phi$, so there is $w' \in W'$ such that $\mo{M}', w' \not \Vdash \phi$. Given
  that we have assumed $\langle Z_1, Z_2 \rangle$ to be surjective with respect
  to $\mo{M}'$, there must be $w\in W$ such that $wZ_1w'$. By the preservation
  of the formulas of $\biLan$ under directed bisimulations, we must have that
  $\mo{M}, w \not \Vdash \phi$ and thus $\mo{M} \not \Vdash \phi$. 
\end{proof}

Let us write $\biItp_\mo{M}(e)$ for the `bi-intuitionistic type' of a point
$e$ in a model $\mo{M}$, i.e., the set of all first order translations of
bi-intuitionistic formulas that $e$ satisfies. If $\mo{M}_1, \mo{M}_2$ are
models, we will write $\mo{M}_1, w \Rrightarrow_\biLan \mo{M}_2, v$ if
$\mo{M}_1, w \Vdash \phi$ implies $\mo{M}_2, v \Vdash \phi$, for
every formula $\phi$ of $\biLan$. If $\mo{M}_1, w \Rrightarrow_\biLan
\mo{M}_2, v$ and $\mo{M}_2, v \Rrightarrow_\biLan \mo{M}_1, w$ both hold, we
write $\mo{M}_1, w \equiv_\biLan \mo{M}_2, v$ and say that
$\mo{M}_1, w$ and $\mo{M}_2, v$ are directed bisimulation equivalent.
Clearly, if $\mo{M}_1, w \equiv_\biLan \mo{M}_2, v$ then they satisfy
precisely the same bi-intuitionistic formulas.

\begin{proposition} \label{pro:hm}  Let $\mo{M}_1$ and $\mo{M}_2$ be two Kripke
models. Suppose that $\mo{M}_1$ and $\mo{M}_2$ are $\omega$-saturated as first
order models.  Then the relation $\Rrightarrow_\biLan$ induces a directed
bisimulation $\langle Z_1, Z_2 \rangle$ between $\mo{M}_1$  and $\mo{M}_2$
defined as follows:  
\begin{align*}
xZ_1 y &\iff \biItp_{\mo{M}_1}(x) \subseteq \biItp_{\mo{M}_2}(y),\\
xZ_2 y &\iff \biItp_{\mo{M}_2}(x) \subseteq \biItp_{\mo{M}_1}(y).
\end{align*}
\end{proposition}

\begin{proof} In what follows let $\{i, j\} = \{1, 2\}$. By definition, if $xZ_i
  y$, i.e., $\biItp_{\mo{M}_i}(x) \subseteq \biItp_{\mo{M}_j}(y)$, we have that
  $\mo{M}_i, w \Vdash \phi$ only if $\mo{M}_j, u \Vdash \phi$  for every formula
  $\phi$ of $\biLan$, and, in particular, $\mo{M}_i, w \Vdash p$ only if
  $\mo{M}_j, u \Vdash p$ for every propositional variable.

Now suppose that $xZ_iy$, i.e., $\biItp_{\mo{M}_i}(x) \subseteq
\biItp_{\mo{M}_j}(y)$, and  $y\leq_j b$ for some $b $ of $ \mo{M}_j$. Consider 
$$
\nbiItp_{\mo{M}_j}(y) = \{\neg ST_x(\psi) :  \mo{M}_j, y \nVdash \psi, \psi \in\biLan\}.
$$
We claim that the set of formulas
$\biItp_{\mo{M}_j}(b) \cup \nbiItp_{\mo{M}_j}(b)$ is satisfiable in $\mo{M}_i$
by an element $b^{\prime}$ such that $x\leq_i b^{\prime}$. Take any finite
subset $S$ of $\biItp_{\mo{M}_j}(b) \cup \nbiItp_{\mo{M}_j}(b)$. Say,
$\{ST_z(\delta_{1}), \dots, ST_z(\delta_{n})\} \subseteq \biItp_{\mo{M}_j}(b)$,
$\{ \neg ST_z(\sigma_{1}), \dots , \neg ST_z (\sigma_{m})\} \subseteq
\nbiItp_{\mo{M}_j}(b)$, and
$$
S = \{ST_z(\delta_{1}), \dots , ST_z(\delta_{n})\} \cup \{ \neg ST_z(\sigma_{1}),
\dots , \neg ST_z (\sigma_{m})\}.
$$ 
It is then clear that
$\mo{M}_j, y \nVdash \bigwedge \{\delta_{1}, \dots , \delta_{n}\} \rightarrow
\bigvee \{ \sigma_{1}, \dots , \sigma_{m}\}$, and thus,
$\mo{M}_j \nvDash ST_z(\bigwedge \{\delta_{1}, \dots , \delta_{n}\} \rightarrow
\bigvee \{ \sigma_{1}, \dots , \sigma_{m}\}) [y]$. Given that
$\biItp_{\mo{M}_i}(x) \subseteq \biItp_{\mo{M}_j}(y)$, we get
$\mo{M}_i \nvDash ST_z(\bigwedge \{\delta_{1}, \dots , \delta_{n}\} \rightarrow
\bigvee \{ \sigma_{1}, \dots , \sigma_{m}\}) [x]$. It follows that
$S$ is satisfiable in $\mo{M}_i$ by an element
$b_0$ such that $x\leq_ib_0$. By the $\omega$-saturation of $\mo{M}_i$, there must
be an element $b^{\prime}$ such that $x\leq_i b^{\prime}$ realising the whole of
$\biItp_{\mo{M}_j}(b) \cup \nbiItp_{\mo{M}_j}(b)$. Since $\biItp_{M_j}(b)$ is
realised by $b^{\prime}$, we have that $bZ_j b^{\prime}$. Since $b^{\prime}$
realises $\nbiItp_{\mo{M}_j}(b)$, that is, $\mo{M}_j, b \nVdash \psi $ only if
$\mo{M}_i, b^{\prime} \nVdash \psi $, by contraposing we obtain that
$\biItp_{\mo{M}_i}(b^{\prime}) \subseteq \biItp_{\mo{M}_j}(b)$, that is,
$b^{\prime}Z_i b$.

On the other hand,  suppose that $xZ_iy$, that is, $\biItp_{\mo{M}_i}(x) \subseteq
\biItp_{\mo{M}_j}(y)$, and  $b\leq_ix$, for some $b$ from $\mo{M}_i$.  
We claim that the set of formulas $\biItp_{\mo{M}_i}(b) \cup
\nbiItp_{\mo{M}_i}(b)$ is satisfiable in $\mo{M}_j$ by an element $b^{\prime}$
such that $b^{\prime}\leq_j y$. As before, take any finite subset $S$ of
$\biItp_{\mo{M}_i}(b) \cup \nbiItp_{\mo{M}_i}(b)$. Say,
$\{ST_z(\delta_{1}), \dots , ST_z(\delta_{n})\}  \subseteq
\biItp_{\mo{M}_i}(b)$, $\{  \neg ST_z(\sigma_{1}), \dots , \neg  ST_z
(\sigma_{m})\} \subseteq \nbiItp_{\mo{M}_i}(b)$, and
$$
S = \{ST_z(\delta_{1}), \dots ,
ST_z(\delta_{n})\} \cup \{  \neg ST_z(\sigma_{1}), \dots ,  \neg ST_z
(\sigma_{m})\}.
$$
It is clear that  $\mo{M}_i, x
\Vdash \bigwedge \{\delta_{1}, \dots , \delta_{n}\} \bito \bigvee \{ \sigma_{1},
\dots ,  \sigma_{m}\}$,  
so given that $\biItp_{\mo{M}_i}(x) \subseteq \biItp_{\mo{M}_j}(y)$, we obtain
$\mo{M}_j, y \Vdash  \bigwedge \{\delta_{1}, \dots , \delta_{n}\} \bito \bigvee \{
\sigma_{1}, \dots ,  \sigma_{m}\}$. It follows that $S$
is satisfiable in $\mo{M}_j$ by an element $b_0$ such that
$b_0\leq_j y$. By the $\omega$-saturation of $\mo{M}_j$, there must be an
element $b^{\prime}$  such that $b^{\prime}\leq_j  y$ realising the whole of
$\biItp_{\mo{M}_i}(b) \cup \nbiItp_{\mo{M}_i}(b)$. Since $\biItp_{\mo{M}_i}(b)$
is realised by $b^{\prime}$, we have that $bZ_i b^{\prime}$ and since
$b^{\prime}$ realises  $\nbiItp_{M_i}(b)$,  we have $ b^{\prime} Z_j b$. 
\end{proof}

\subsection{Definable prime filter extensions}\label{ssec:def-p-filt-ext} 

Recall from Section~\ref{ssec:duality} the double-dual model  $(\mo{M}^+)_+ =
\bigl((\mo{F}^+)_+, V^*)$, with $\mo{F} = (W,\leq)$.  
We will call $(\mo{M}^+)_+$ a \emph{prime filter extension} of $\mo{M}$ and
denote it by $\mathfrak{pe}(\mo{M})$. This is because the
mapping $\pi : W \longrightarrow  \mathcal{F}_p(\mathrm{Up}(W))$ defined by 
$\pi(w) = \{x \in \mathrm{Up}(W): w \in x\}$ is an embedding, in the standard
model-theoretic sense, of the structure $\mo{M}$ into  $\mathfrak{pe}(\mo{M})$. 

For $x, y \subseteq\mathrm{Up}(W)$ we will say that $x$ \emph{is separated from}
$y$ if for any finite subsets $x' \subseteq x$ and $y' \subseteq y$ we have that
$\bigcap x' \not \subseteq \bigcup y'$. The following version of the prime
filter theorem is 
well known.

\begin{lemma}\label{lem:prime-fil-thm}
If $x$ is separated from $y$, then $x$ is included in a prime filter of\/
$\mathrm{Up}(W)$ that is disjoint from $y$.
\end{lemma}

\begin{proposition} \label{inpe}
Let $\mo{M} = (W, \leq, V)$ be a Kripke model for bi-intuitionistic logic. Then
for any $u\in\mathcal{F}_p(\mathrm{Up}(W))$ and any formula $\phi$ of\/ $\biLan$
the following equivalences hold:
\begin{enumerate}
\item $V(\phi) \in u \iff \mathfrak{pe}(\mo{M}), u \Vdash \phi$.
\item $\mo{M}, w \Vdash \phi \iff \mathfrak{pe}(\mo{M}), \pi(w) \Vdash \phi$ for any
 $\phi\in \biLan$.
\item $\mo{M} \Vdash \phi \iff \mathfrak{pe}(\mo{M}) \Vdash \phi$ for any
  $\phi\in\biLan$. 
\end{enumerate}
\end{proposition} 

\begin{proof} We establish (1)  by induction on the complexity of $\phi$. The
  case when $\phi$ is a propositional variable is immediate by definition of
  $V^{\mathfrak{pe}(\mo{M})}$. The case when $\phi = \bot$ is trivial since
  $V(\bot) = \emptyset$, which can never belong to a filter $u$. The cases for
  $\vee$ and $\wedge$ are obvious.

Let $\phi = \psi \rightarrow \chi$. Suppose that $\mathfrak{pe}(\mo{M}), u
\nVdash  \psi \rightarrow \chi$, i.e., there is $u_1$ such that
$u\leq^{\mathfrak{pe}(\mo{M})}u_1$ while $\mathfrak{pe}(\mo{M}), u_1 \Vdash
\psi $ and $\mathfrak{pe}(\mo{M}), u_1 \nVdash   \chi$.
By inductive hypothesis, $V(\psi) \in u_1$ and $V(\chi) \notin u_1$.
Hence $V(\psi) \rightarrow V(\chi) = V(\psi \rightarrow \chi) \notin u$ by
definition of $\leq^{\mathfrak{pe}(\mo{M})}$. On the other hand, suppose that
$\mathfrak{pe}(\mo{M}), u \Vdash \psi \rightarrow \chi$. We observe that $u \cup
\{V(\psi)\}$ cannot be separated from $\{V(\chi)\}$. Otherwise, there would be
$v \in \mathcal{F}_p(\mathrm{Up}(W))$ such that  $u \subseteq v, V(\psi) \in v,$ and
$V(\psi) \notin v$. Since $u \subseteq v$, by Lemma~\ref{lem:modus-ponens} it
follows that $u\leq^{\mathfrak{pe}(\mo{M})}v$ as well.
By inductive hypothesis, this would
imply that $\mathfrak{pe}(\mo{M}), v \Vdash \psi$ and $\mathfrak{pe}(\mo{M}), v
\nVdash \chi$, contradicting  $\mathfrak{pe}(\mo{M}), u \Vdash \psi \rightarrow
\chi$. Thus $u \cup \{V(\psi)\}$ cannot be separated from $\{V(\chi)\}$ and
since $u$ is closed under finite intersections there must be some $y \in u$ with
$y \cap V(\psi) \subseteq V(\chi)$, which entails that  $y \subseteq V( \psi
\rightarrow \chi)$ and hence  $V(\psi \rightarrow \chi) \in u$ as $u$ is a
filter.

Suppose $V(\psi \bito \chi) \notin u$.
Let $u' \in \mathcal{F}_p(\mathrm{Up}(W))$ be any prime filter such that $u' \subseteq u$ and  
suppose $\mathfrak{pe}(\mo{M}), u' \Vdash \psi$.
Then by the induction hypothesis we have $V(\psi) \in u'$.
Since $V(\psi) \subseteq V(\chi) \cup V(\psi \bito \chi)$
we find $V(\chi) \cup V(\psi \bito \chi) \in u'$.
We have $V(\psi \bito \chi) \notin u'$ because $u' \subseteq u$,
so since $u'$ is prime it must be the case that $V(\chi) \in u'$.
The induction hypothesis then implies $\mathfrak{pe}(\mo{M}), u' \Vdash \chi$.
By Lemma~\ref{lem:modus-ponens}, it implies that no prime filter $u'$ below $u$
can satisfy $ \psi $ but not $\chi$, and therefore
$\mathfrak{pe}(\mo{M}), u \nVdash  \psi \bito \chi$.
   For the converse, suppose $\mathfrak{pe}(\mo{M}), u \nVdash \psi \bito \chi$,
and let $I = \mathrm{Up}(W) \setminus u$. Then, by well-known properties of
distributive lattices, $I$ is a prime ideal.
Suppose towards a contradiction that there is no $a \in I$ such that
 $V(\psi) \subseteq V(\chi) \cup a$.
  Then $\{ V(\psi) \}$ is separated from $I \cup \{ V(\chi) \}$,
  so there exists a prime filter $u'$ containing $V(\psi)$ which is
  disjoint from $I \cup \{ V(\chi) \}$.%
  This implies $V(\psi) \in u'$ and $V(\chi) \notin u'$,
  so by the induction hypothesis we have
  $\mathfrak{pe}(\mo{M}), u' \Vdash \psi$ while
  $\mathfrak{pe}(\mo{M}), u' \nVdash \chi$.
  But since $u' \subseteq \mathrm{Up}(W) \setminus I = u$,
  this contradicts the assumption that $u \not\Vdash \psi \bito \chi$.
  So there must be some $a \in I$ such that
  $V(\psi) \subseteq V(\chi) \cup a$.
  This implies $V(\psi \bito \chi) \subseteq a$.
  Since $a \in I$ we have $a \notin u$, and hence $V(\psi \bito \chi) \notin u$.

For part (2) suppose that  $\mo{M} \Vdash \phi$. Then $V(\phi)= W$ and hence
$V(\phi) \in u$ for every $u \in  \mathcal{F}_p(\mathrm{Up}(W))$, so
$\mathfrak{pe}(\mo{M}),u \Vdash \phi$. Thus 
$\mathfrak{pe}(\mo{M}) \Vdash \phi$ since $u$ was  arbitrary. Conversely
if $\mo{M} \nVdash \phi$, there is some $w \in W$ such that $\mo{M}, w \nVdash
 \phi$. Thus, $\mathfrak{pe}(\mo{M}), \pi(w) \nVdash \phi$, and $\mathfrak{pe}(\mo{M})
 \nVdash \phi$ as desired. 
\end{proof}

\begin{definition}
Let $\mo{M} = (W, \leq, V)$ be a model. The \emph{definable dual algebra}
$\mo{M}^{+ \delta}$ of $\mo{M}$ is the structure $\langle
U(\mo{M}), \rightarrow, \bito, \cap, \cup, 1, 0 \rangle$,  where
$U(\mo{M}) = \{V(\phi) : \phi \in \biLan\}$ and the operations are
defined as in Definition~\ref{def:upset-alg}. In particular, 
$V(\top) = W$ and $V(\bot) = \emptyset$. 
\end{definition}

\begin{proposition} Let $\mo{M}$ be a Kripke model for bi-intuitionistic
logic. Then $\mo{M}^{+ \delta}$ is a double Heyting algebra. Indeed,
$\mo{M}^{+ \delta}$ is a subalgebra of\/ $\mo{M}^{+}$ generated by the set\/
$\{V(p): p\in\Prop\}$.
\end{proposition}

\begin{definition}
  Let $\mo{M} = (W, \leq, V)$ be a model.  The \emph{dual model}
  $(\mo{M}^{+  \delta})_+$ of\/ $\mo{M}^{+ \delta}$ is the structure $\langle
\mathcal{F}_p(U(\mo{M})), \leq^*,  V^* \rangle$  where: 
\begin{enumerate}
\item $\mathcal{F}_p(U(\mo{M}))$ is the set of all prime filters of
  $U(\mo{M})$ in the algebra $\mo{M}^{+ \delta}$, 
\item $w\leq^* v$ iff for all $x, y \in U(\mo{M})$, if $x
  \rightarrow y \in w$  and $x \in v$ then $y \in v$, 
\item $V^*(p) = \{ w \in \mathcal{F}_p(U(\mo{M})) : V(p) \in w\}$.
\end{enumerate}
\end{definition}

We will write $\mo{M}^{\delta}$ for $(\mo{M}^{+ \delta})_+$, and call it the
\emph{definable prime filter extension} of $\mo{M}$.
It is easy to see that $\mo{M}$ embeds into $\mo{M}^{\delta}$ via the map 
$\pi : W \longrightarrow  \mathcal{F}_p(U(\mo{M}))$ given by
$$
\pi(w) = \{x \in U(\mo{M}): w \in x\}.
$$
By Proposition~\ref{prop:min-dual-alg}(3) it follows
that $\mo{M}^\delta$ is a bounded-morphic image of $\mathfrak{pe}(\mo{M})$, but
we will prove it directly below.

\begin{proposition}\label{prob} Let $\mo{M}$ be a Kripke model for
  bi-intuitionistic logic. The map $x \mapsto x \cap U(\mo{M})$ is a bounded
  morphism from $\mathfrak{pe}(\mo{M})$ onto $\mo{M}^{\delta}$. Moreover, for
  each $x \in \mathcal{F}_p(U(\mo{M}))$ there is $y \in \mathcal{F}_p(\mathrm{Up}(W))$
  such that $x = y \cap U(\mo{M})$.
\end{proposition}

\begin{proof} Consider the inclusion homomorphism $i$ from $\mo{M}^{+ \delta}$ into $\mo{M}^+$. A dual mapping $i_+: (\mo{M}^+)_+ \longrightarrow (\mo{M}^{+ \delta})_+$ is defined by the equation:
\begin{center}
$i_+(x) = i^{-1}(x)$.
\end{center}
But  $i^{-1}(x) = \{y \in U(\mo{M}) : y \in x\}$, so $i_+$ is indeed the map
mentioned in the statement of the proposition. Furthermore, $i_+$ is a bounded
morphism. The only thing that we need to notice is that $\mathfrak{pe}(\mo{M}),
x \Vdash p$ iff $V(p) \in x$ iff $V(p) \in  x \cap U(\mo{M})$ iff
$\mo{M}^{\delta},  x \cap U(\mo{M}) \Vdash p $, for any propositional variable
$p$  of  $\biLan$. 

Finally let $x \in \mathcal{F}_p(U(\mo{M}))$. Since $x$ is a prime filter of
$U(\mo{M})$ it is not difficult to see that $x$ is separated from $U(\mo{M})
\setminus x$. Hence, by Lemma~\ref{lem:prime-fil-thm}, there is a prime filter
$y\subseteq U(\mo{M})$, such that $y \supseteq x$ and
$y \cap (U(\mo{M})\setminus x) = \emptyset$.
It follows that $ x= y \cap U(\mo{M})$, as desired.  A similar argument
establishes that the mapping under consideration is surjective. 
 \end{proof}

 We will say that  a class $K$ is  \emph{closed under surjective   directed
   bisimulations}   if whenever  
 $\mo{M} = (W, \leq, V)$ and $\mo{M}' = (W', \leq', V')$ are two
  Kripke models and  $\langle Z_1, Z_2 \rangle$  (where $Z_1\subseteq W\times
  W'$ and $Z_2\subseteq W'\times W$) is a  directed bisimulation surjective with
  respect to $\mo{M}'$, then $\mo{M} \in K$ only if $\mo{M}' \in K$.

\begin{corollary} \label{clo}
Suppose K is a class of  Kripke models for bi-intuitionistic logic closed under
surjective   directed bisimulations. Then $\mathfrak{pe}(\mo{M}) \in K$ if and
only if\/ $\mo{M}^{\delta} \in K$.
\end{corollary}

\begin{proof}
Let $g$ be the bonded morphism $ x \mapsto x \cap U(\mo{M})$,  given in
Proposition \ref{prob}. By Proposition \ref{bmdr},  the pair $\langle Z_1,
Z_2\rangle$  is a directed bisimulation which is surjective with respect to both   $\mo{M}^{\delta}$ and  $\mathfrak{pe}(\mo{M})$, where:
\begin{align*}
xZ_1y &\iff g(x) \leq^{\mo{M}^{\delta}} y,\\
xZ_2y &\iff x \leq^{\mo{M}^{\delta}} g(y).
\end{align*}
Consequently, $\mathfrak{pe}(\mo{M}) \in K$ iff  $\mo{M}^{\delta} \in K$ by the
closure assumption on $K$. 
\end{proof}

\begin{corollary} \label{dpfe}
Let $\mo{M}$ be a Kripke model for bi-intuitionistic logic. Then for any prime
filter $u$ of $U(\mo{M})$ and formula $\phi$ of\/ $\biLan$ we have that
$V(\phi) \in u$ if and only if $\mo{M}^{\delta}, u \Vdash \phi$.
Moreover, $\mo{M} \Vdash \phi$ if and only if $\mo{M}^{\delta} \Vdash \phi$ for
any such $\phi$. 
\end{corollary}

\begin{proof}
Consider the mapping given in Proposition \ref{prob}. We have that
$\mathfrak{pe}(\mo{M}), x \Vdash \phi$ iff  $\mo{M}^{\delta}, x \cap U(\mo{M})
\Vdash \phi$ and that indeed all elements of $\mo{M}^{\delta}$ are of the form $
x \cap U(\mo{M}) $. Therefore, using  Proposition \ref{inpe}, we see that if $u=
x \cap U(\mo{M})$ is a prime filter of $U(\mo{M})$, $V(\phi) \in u$ iff $V(\phi)
\in x$ iff $\mathfrak{pe}(\mo{M}), x \Vdash \phi$ iff $\mo{M}^{\delta}, u \Vdash
\phi$. 

The last part of the result follows from  Proposition \ref{inpe} again using
Corollary \ref{clo}. 
\end{proof}

According to Corollaries~\ref{rdbp}, \ref{dpfe} and Proposition~\ref{inpe}, we
can see that a class of models $K$ definable by a theory of $\biLan$  is always
going to be closed under surjective directed bisimulations, (definable) prime
extensions and disjoint unions. Indeed, both $K$ and its complement
$\overline{K}$ will be closed under (definable) prime extensions.

\begin{definition}\label{def:inner-subm}
Let $\mo{M}$ be a model. A submodel (in the classical first-order sense)
$\mo{M}^{\prime}$ of\/ $\mo{M}$ will be called
an \emph{inner submodel} of\/ $\mo{M}$ if the pair $\langle I, I\rangle$, where
$I$ is the identity relation on $\mo{M}^{\prime}$, is a  directed bisimulation between
$\mo{M}$ and $\mo{M}^{\prime}$.
\end{definition}

Clearly, if $\mo{M}^{\prime}$ is an inner submodel of $\mo{M}$,
then $\langle I, I \rangle$ is a directed bisimulation surjective with respect to
$\mo{M}^{\prime}$. Note that, using Corollary~\ref{rdbp}, it follows that a
definable class of models is always going to be closed under inner submodels,
that is, given a class $K$, if $\mo{M} \in K$ and $\mo{M}^{\prime}$ is an inner
submodel of $\mo{M}$, then $\mo{M}^{\prime} \in K$. 

\begin{lemma} \label{gold}
Let $\{\mo{M}_i : i \in I\}$ be a family of Kripke models and
$\prod_{i\in I} \mo{M}_i /U$ an ultraproduct. Then  $\prod \mo{M}_i / U$ is
isomorphic to an inner submodel of the ultrapower
$(\coprod_{i \in I} \mo{M}_i)^I/U$. Hence, a class
closed under disjoint unions, inner submodels and ultrapowers is closed under
ultraproducts.
\end{lemma}

\begin{proof} Simply consider the mapping $f/ U \mapsto f^{\prime}/ U$ where $
  f^{\prime}: I \longrightarrow \bigcup_{i \in I} W_i $ is defined by
  $f^{\prime}(i) = f(i)$.
\end{proof}

\begin{lemma}\label{comp}
Let $\biLan^{corr +}$ be obtained by adding a list of constants to
$\biLan^{corr}$. Suppose $K$ is a class of Kripke models which is closed under
ultraproducts. Then for any  set $\Theta$ of formulas of\/ $\biLan^{corr +}$, if
$\Theta$ is finitely satisfiable in $K$ then $\Theta$ is satisfiable in
$K$.
\end{lemma}

\begin{proof}
This is simply the proof of the compactness theorem using ultraproducts which
can be found for example in~\cite{bell} (Theorem 4.1). 
\end{proof}
If $\mo{M} = (W, \leq, V)$ is a model and $x\in W$, we consider the inner
submodel $\mo{M}_x $ with domain
$$
W_x= \{y \mid \exists n\in\mathbb{N}\ \exists z_0,  \dots, z_n
\bigl(\bigwedge_{i < n}(z_{i+1} \leq z_i \ \vee \ z_i \leq z_{i+1}) \wedge (z_n = y)
\wedge (z_0 = x)\bigr)\},
$$
and where both
the ordering and the valuation are induced from $\mo{M} = (W, \leq, V)$.

\begin{lemma} \label{disjoint} Let $\mo{M} = (W, \leq, V)$ be a Kripke model for
  bi-intuitionistic logic and $K$ a class of models closed under directed
  bisimulations and disjoint unions, then 
  $$
  \mo{M} \in K \iff \mo{M}_x \in K,  \text{ for all } x\in W.
  $$ 
 \end{lemma}

\begin{proof} First, we can take isomorphic copies of the models $\mo{M}_x$ ($x
  \in W$) to make sure that they're pairwise disjoint. For definiteness, we can
  simply let the domain of $\mo{M}_x$ be the set $\{\langle w, x\rangle \mid w
  \in W_x\}$.  Notice that for the disjoint union $ \coprod_{x \in W}
  \mo{M}_x$, the mapping $\langle w, x\rangle \mapsto w$ is a surjective
  bounded morphism (cf. the remark after Definition~\ref{def:disj-union}).
  This suffices to establish our result. 
\end{proof}

\begin{proposition}\label{tdb} Let $\mo{M}= (W, \leq, V)$ be a   Kripke model
  and $\mo{N}=(W', \leq', V')$ an $\omega$-saturated   Kripke model such that
$$
\mo{N} \Vdash \phi \iff \mo{M} \Vdash \phi
$$
for any formula $\phi\in \biLan$. Then there is a directed bisimulation surjective
with respect to both $\mo{N}$ and $\mo{M}^{\delta}$ between these two
 models.
\end{proposition} 

\begin{proof} Given $x \in W'$, let
\begin{center}
$f(x) = \{V(\phi): \phi \in \biLan \mathbin{\&} \mo{N}, x \Vdash \phi\}$.
\end{center}
It is not difficult to verify that $f(x)$ is indeed in the domain of $\mo{M}^{\delta}$.
Now define the following pair of relations:
\begin{align*}
x Z_1 y &\iff x \subseteq f(y),\\
x Z_2 y &\iff f(x) \subseteq y.
\end{align*}
Next we show that $\langle Z_1, Z_2 \rangle$ is a directed bisimulation
surjective with respect to both $\mo{N}$ and $\mo{M}^{\delta}$.

Assume that $x Z_1 y$, that is, $x \subseteq f(y)$. Now, $\mo{M}^{\delta}, x
\Vdash p$ implies that $V(p) \in x$, so by assumption, $V(p) \in  f(y)$ as
well. The latter means that  $\mo{N}, y \Vdash p$ by definition of $f$. On the
other hand if $x Z_2 y$, i. e.,  $f(x) \subseteq y$ and $\mo{N}, x \Vdash p$,
then $V(p) \in f(x)$. Consequently,  $V(p) \in y$, by our assumption, and that
means that  $\mo{M}^{\delta}, y \Vdash p$ as desired.

For the next condition, suppose that $x Z_1 y$, that is, $x \subseteq f(y)$ and
$y\leq' b$. Take $f(b)$. It is easy to see that $R^{\delta}xf(b)$. For assume
that $V(\phi) \rightarrow V(\psi) = V(\phi \rightarrow \psi) \in x$ while
$V(\phi) \in f(b)$. Then $\mo{N}, b \Vdash \phi$ whereas $V(\phi \rightarrow
\psi) \in f(y)$, which means that $\mo{N}, y \Vdash \phi \rightarrow \psi$, so
$\mo{N}, b \Vdash  \psi$ given that  $y\leq' b$. Thus, $V(\psi) \in f(b)$ as
desired. Moreover, $bZ_2f(b)$ while $f(b)Z_1b$. On the other hand assume that $x
Z_2 y$, that is,   $f(x) \subseteq y$ and $R^{\delta}yb$. Consider the set
$\Delta$ defined as 
$$
\{ ST_x(\phi) : \phi \in \biLan  \mathbin{\&} \mo{M}^{\delta}, b \Vdash \phi\}
\cup \{\neg
ST_x(\psi) : \psi \in \biLan \mathbin{\&} \mo{M}^{\delta}, b \nVdash \psi \}. 
$$
Take any finite $\Delta_0 \subseteq \Delta$. We may assume
$$
\Delta_0 = \{ST_x(\phi_0), \dots, ST_x(\phi_j)\} \cup \{\neg ST_x(\psi_0),
\dots, \neg ST_x(\psi_k)\}.
$$
Now, $\mo{M}^{\delta}, y \nVdash \bigwedge_{i < j+1} \phi_i \rightarrow
\bigvee_{i<k+1} \psi_i $, so $V(\bigwedge_{i < j+1} \phi_i \rightarrow
\bigvee_{i<k+1} \psi_i) \notin y$, and therefore
$V(\bigwedge_{i < j+1} \phi_i \rightarrow \bigvee_{i<k+1} \psi_i) \notin
f(x)$.
This means that $\mo{N}, x \nVdash \bigwedge_{i < j+1} \phi_i
\rightarrow \bigvee_{i<k+1} \psi_i$,  hence, there must be $b_0$ such that
$x\leq' b_0$ and $\mo{N}, b_0 \Vdash \bigwedge_{i < j+1} \phi_i $ while $\mo{N},
b_0 \nVdash \bigvee_{i<k+1} \psi_i $. By the $\omega$-saturation of $\mo{N}$ we
must have $b^{\prime}$ such that $x\leq' b^{\prime}$ while also $b^{\prime}$ 
satisfies
$\{ ST_x(\phi):  \phi \in \biLan \mathbin{\&}\mo{M}^{\delta}, b \Vdash \phi\}$
(that is, $\mo{M}^{\delta}, b \Rrightarrow_\biLan \mo{N}, b^{\prime}$) and
$b^{\prime}$ satisfies
$\{\neg ST_x(\psi):\psi\in\biLan\mathbin{\&}\mo{M}^{\delta},b \nVdash\psi\}$
(that is, $  \mo{N}, b^{\prime} \Rrightarrow_\biLan\mo{M}^{\delta}, b$ ).
Finally, $b \subseteq f(b^{\prime})$, that is $b Z_1 b^{\prime}$, since
$V(\phi) \in b$ implies that $\mo{M}^{\delta}, b \Vdash \phi$ which in turn
means that $\mo{N}, b^{\prime}\Vdash \phi$, so indeed
$V(\phi) \in f(b^{\prime})$. Similarly, $f(b^{\prime}) \subseteq b$,
that is $b^{\prime} Z_2b$, since
$\mo{N}, b^{\prime} \Rrightarrow_\biLan \mo{M}^{\delta}, b$.

If $xZ_iy$ and $b\leq_ix$ for some $b\in W_i$,   there is $b^{\prime}\in W_j$ such
that $b^{\prime}\leq_jy$, $bZ_i b^{\prime}$ and $b^{\prime}Z_jb$.

Now suppose that $x Z_1 y$, that is, $x \subseteq f(y)$ and
$b\leq x$ for some 
$b$ of $\mo{M}^\delta$. We must find a $b^{\prime}\in W'$ such that
$b^{\prime}\leq' y$, $bZ_1 b^{\prime}$ and $b^{\prime}Z_2b$.
Consider the set $\Delta$ defined as 
$$
\{ ST_x(\phi):\phi\in\biLan\mathbin{\&}\mo{M}^{\delta}, b\Vdash\phi\} \cup
\{\neg ST_x(\psi):\phi\in\biLan\mathbin{\&}\mo{M}^{\delta}, b \nVdash \psi\}.
$$
Take any finite $\Delta_0 \subseteq \Delta$. We may assume
$$
\Delta_0 = \{ST_x(\phi_0), \dots, ST_x(\phi_j)\} \cup
\{\neg ST_x(\psi_0), \dots , \neg ST_x(\psi_k)\}.
$$
Now, $\mo{M}^{\delta}, x \Vdash \bigwedge_{i < j+1} \phi_i \bito \bigvee_{i<k+1}
\psi_i $, so $V(\bigwedge_{i < j+1} \phi_i \bito \bigvee_{i<k+1} \psi_i) \in x$,
so $V(\bigwedge_{i < j+1} \phi_i \bito \bigvee_{i<k+1} \psi_i) \in f(y)$. The
latter means then that $\mo{N}, y \Vdash \bigwedge_{i < j+1} \phi_i \bito
\bigvee_{i<k+1} \psi_i$, hence, there must be $b_0$ such that $b_0\leq' y$ and
$\mo{N}, b_0 \Vdash \bigwedge_{i < j+1} \phi_i $ while
$\mo{N}, b_0 \nVdash\bigvee_{i<k+1} \psi_i$.
By the $\omega$-saturation of $\mo{N}$ we must have
$b^{\prime}$ such that $b^{\prime} \leq' y $ while also $b^{\prime}$ satisfies
$\{ST_x(\phi): , \phi \in \biLan\mathbin{\&}\mo{M}^{\delta}, b \Vdash \phi\}$
(that is, $\mo{M}^{\delta}, b \Rrightarrow_\biLan \mo{N}, b^{\prime}$)
and $b^{\prime}$ satisfies
$\{\neg ST_x(\psi):\psi\in\biLan\mathbin{\&} \mo{M}^{\delta}, b \nVdash \psi\}$
(that is, $\mo{N}, b^{\prime} \Rrightarrow_\biLan \mo{M}^{\delta}, b$).
Finally, $b \subseteq f(b^{\prime})$, that is, $b Z_1 b^{\prime}$ since
$V(\phi) \in b$ implies that $\mo{M}^{\delta}, b \Vdash \phi$ which means that $
\mo{N}, b^{\prime}\Vdash \phi$, so indeed $V(\phi) \in f (b^{\prime})$, and
$f(b^{\prime}) \subseteq b$, that is, $b^{\prime} Z_2b$ similarly since $  \mo{N},
b^{\prime} \Rrightarrow_\biLan \mo{M}^{\delta}, b$.

Finally, we show that $\langle Z_1, Z_2 \rangle$ is surjective with respect to
both $\mo{N}$ and $\mo{M}^{\delta}$. Assuming first that $x \in W'$, we have 
that $f(x) Z_1 x$ trivially. On the other hand, if $x$ is a world
in $ \mo{M}^{\delta}$ then it suffices to find $y \in \mo{N}$ such that
$\mo{M}^{\delta}, x \nVdash \phi$ implies that $\mo{N}, y \nVdash \phi$ for all
formulas $\phi$ of $\biLan$. This will show that $y Z_2 x$.
Consider the
set $ \{\neg ST_x(\phi):\phi \in \biLan\mathbin{\&} M^{\delta}, x \nVdash \phi\}$
and take some finite subset $\{\neg ST_x(\phi_0), \dots , \neg ST_x(\phi_n)\}$.
We know that $\mo{M}^{\delta}, x \nVdash \bigvee_{j < n+1}\phi_j $, so
$\mo{M}^{\delta}\nVdash \bigvee_{j < n+1}\phi_j $, so
$\mo{M} \nVdash \bigvee_{j < n+1}\phi_j$ and by the hypothesis of the
proposition, $\mo{N} \nVdash \bigvee_{j < n+1}\phi_j$. Consequently, there is
$z \in \mo{N}$ such that $\mo{N}, z \nVdash \bigvee_{j < n+1}\phi_j$, so $z$
satisfies $\{\neg ST_x(\phi_0), \dots , \neg ST_x(\phi_n)\}$. By the
$\omega$-saturation of $\mo{N}$, there must be $y$ satisfying
$ \{\neg ST_x(\phi):\phi\in\biLan\mathbin{\&}\mo{M}^{\delta},x\nVdash\phi\}$
as desired.
 \end{proof}

\begin{lemma}\label{ul}  Let K be a class of  Kripke models for
bi-intuitionistic logic. If K is closed under surjective  directed
bisimulations and both K and $\overline{K}$ are closed under definable prime filter
extensions, then both  K and its complement  are closed under ultrapowers.
\end{lemma} 

\begin{proof}
First suppose $\mo{M} \in K$ and $\prod \mo{M}/U$ is an ultrapower
of $\mo{M}$. Consider next an ultrapower $\mo{N}=\prod (\prod \mo{M}/U)/D$
obtained by the Keisler method which is $\omega$-saturated. Applying {\L}o\'s
theorem twice, we get that $\mo{M} \Vdash \phi$ if and only if
$\mo{N}\Vdash \phi$. Thus, using Proposition \ref{tdb}, there is a
surjective (with respect to $\mo{N}$) directed bisimulation between
$\mo{M}^{\delta}$ and $\mo{N}$.  

Now $\mo{M}^{\delta} \in K$ by the closure of $K$ under definable prime filter
extensions. By the closure under surjective
directed bisimulations, we also see that $\mo{N} \in K$. Again applying
Proposition \ref{tdb} with $\mo{N}$ and $\prod \mo{M}/U$ this time, we get
that there is a directed bisimulation surjective with respect to $\prod \mo{M}/U$
between $\mo{N}$ and $(\prod \mo{M}/U)^{\delta}$, so indeed the latter is in $K$
and since $\overline{K}$ is closed under definable prime filter extensions, we must
have that $\prod \mo{M}/U \in K$, as desired.

On the other hand if $\mo{M} \in \overline{K}$, since   also $\mo{M}^{\delta}
\in \overline{K}$ then, by  Proposition \ref{tdb}, we have that $\prod (\prod
\mo{M}/U)/D \in \overline{K}$ by the closure of $K$ under surjective   directed
bisimulations. Moreover, $\prod \mo{M}/U \in \overline{K}$ by the closure of $K$
under ultrapowers established above.
 \end{proof}

\section{Classes of models axiomatisable in $\biInt$}\label{sec:axiomatisable-models} 

In this section we are going to characterise classes of models
\emph{axiomatisable} by theories in  $\biLan$.
The exact meaning of axiomatisability that we employ here will be clarified in
the following definition, where, as usual, for a set $\Theta\subseteq\biLan$, we
write $\mbox{Mod}(\Theta)$ for
$\{ \mo{M}\mid \mo{M} \Vdash \Theta\}$.

\begin{definition}\label{def:axiomatisable-class}
  A class K of Kripke models is said to be
  \emph{axiomatisable} or \emph{definable}  in $\biLan$ if there is a set
  $\Theta$ of formulas of $\biLan$ such that $\mbox{Mod}(\Theta) = K$.  
\end{definition}

\begin{theorem}\label{axmod} Let K be a class of Kripke models for
  bi-intuitionistic logic. Then the following are equivalent: 
\begin{enumerate}
\item $K$ is bi-intuitionistically definable, that is, $K = \mbox{Mod}(\Theta)$
  for some collection $\Theta$ of formulas of $\biLan$. 
\item K is closed under surjective directed bisimulations and disjoint unions
  while both K and its complement $\overline{K}$ are closed under prime filter
  extensions. 
\item K is closed under surjective  directed bisimulations and disjoint unions
  while both K and its complement $\overline{K}$ are closed under definable
  extensions. 
\item K is closed under surjective  directed bisimulations and disjoint unions
  while both K and its complement $\overline{K}$ are closed under ultrapowers. 
\end{enumerate}
\end{theorem}

\begin{proof} 
$(1) \Rightarrow (2)$: Closure under surjective  directed bisimulations comes
from Corollary \ref{rdbp}. Closure under disjoint unions  is a consequence of
previous results. Finally, the remaining closure properties follow from
Proposition \ref{inpe}. 

$(2) \Rightarrow (3)$: By Corollary \ref{clo}, we see that if $K$ is closed
under prime filter extensions, then it is also closed under definable
extensions.  

$(3) \Rightarrow (4)$: By Lemma \ref{ul}. 

$(4) \Rightarrow (1)$: Suppose that $K$ and $\overline{K}$ are closed as indicated. Observe that $K$ is closed under ultraproducts according to  Lemma~\ref{gold}. Let
$$
\mathrm{Th}(K)= \{\phi \in \biLan:
\mo{M} \Vdash \phi, \ \mbox{for all} \ \mo{M}\in K\}.
$$ 
All we need to do is show that $K =  \mbox{Mod}(\mbox{Th}(K))$. The direction $K
\subseteq  \mbox{Mod}(\mbox{Th}(K))$ is obvious, so let $\mo{M}= (W, \leq, V)
\in \mbox{Mod}(\mbox{Th}(K))$ to establish that $\mo{M} \in K$.
Observe that since there is directed bisimulation surjective with respect to
$\mo{M}_x$, by Proposition~\ref{rdbp} we have that $\mo{M} \in \mbox{Mod}(\mbox{Th}(K))$
only if $\mo{M}_x \in \mbox{Mod}(\mbox{Th}(K))$), so    
by Lemma~\ref{disjoint}, we may assume that $\mo{M}$ is of the form $\mo{M}_x$.

Next, consider the following set $\Delta$ of  formulas of  $\biLan^{corr +}$:
$$
\{ ST_x(\phi):\phi \in \biLan\mathbin{\&} \mo{M}, x \Vdash \phi\} \cup
\{\neg ST_x(\phi):\phi \in \biLan\mathbin{\&}\mo{M}, x \nVdash \phi\}.
$$
We will show that $\Delta$ is finitely satisfiable in $K$, which, together with the
assumption that  $K$ is closed under ultraproducts, will yield, 
by appealing to Lemma~\ref{comp}, that indeed $\Delta$ is satisfiable in $K$.
To this end, take any finite $\Delta_0 \subseteq \Delta$. Without loss of generality, we
may assume that 
$$
\Delta_0 = \{ ST_x(\phi_1), \dots, ST_x(\phi_n), \neg ST_x(\psi_1), \dots,  \neg
ST_x(\psi_m)\}.
$$
But then $\mo{M}_x,x \nVdash \bigwedge_{i \leq n} \phi_i \rightarrow \bigvee_{i
  \leq m} \psi_i $, so $ \bigwedge_{i \leq n} \phi_i \rightarrow \bigvee_{i \leq
  m} \psi_i \notin \mbox{Th}(K)$ since $\mo{M}_x \in \mbox{Mod}(\mbox{Th}(K))$.
Therefore, there must be $\mo{M}'\in K$ such that
$\mo{M} \nVdash \bigwedge_{i \leq n} \phi_i \rightarrow \bigvee_{i \leq m} \psi_i $. Thus 
$\Delta $ is finitely satisfiable in $K$, and so $\Delta$ is satisfiable in $K$
by some point $y$ in a  model $\mo{N}= (W', \leq', V')$.
Hence, $\mo{N}, y \equiv_\biLan \mo{M}_x, x$.

Now we may assume that both $\mo{N}$ and $ \mo{M}_x$ are $\omega$-saturated
since if we take $\omega$-saturated ultrapowers  $\mo{N}'$ and $ \mo{M}_x'$ of
$\mo{N}$ and $ \mo{M}_x$ respectively, we have that $\mo{N}'\in K$ (by closure
under ultrapowers) and $ \mo{M}_x' \in \mbox{Mod}(\mbox{Th}(K))$
(since $\mo{M}_x'$ is an elementary extension of $ \mo{M}_x$).
By Proposition \ref{pro:hm}, we have that $\langle Z_1, Z_2 \rangle$ is a
directed bisimulation between $\mo{M}_x$  and $\mo{N}$, defined as follows:  
\begin{align*}
xZ_1 y &\iff \biItp_{ \mo{M}_x}(x) \subseteq \biItp_{\mo{N}}(y),\\
xZ_2 y &\iff   \biItp_{\mo{N}}(x) \subseteq \biItp_{ \mo{M}_x}(y).
\end{align*}

All that remains  to do is show that $\langle Z_1, Z_2 \rangle$  is surjective
with respect to  $\mo{M}_x$. The rest of the proof is laborious but not
particularly difficult. We leave the details to the reader. It suffices to use
the definition of  the domain of $\mo{M}_x$ to check that the relation is
surjective. 
\end{proof}

\section{Goldblatt-Thomason theorem}\label{sec:Goldblatt-Thomason}

We end by proving a bi-intuitionistic analogue of the Goldblatt-Thomason Theorem.
Apart from the result itself, this section also serves an an illustration
of algebraic methods. Most of the preliminary results have already been stated,
we only use one additional lemma.

\begin{lemma}\label{lem:satur}
Let $\mo{F}$ be a Kripke frame. There exists an ultrapower $\mo{F}^I\!/U$ such that
$\mathfrak{pe}(\mo{F})$ is a bounded-morphic image of\/ $\mo{F}^I\!/U$. 
\end{lemma}  

\begin{proof} We only sketch a proof. 
Expand the first-order language of $\mo{F}$ by adding as many unary predicates $P_j$
as there are upsets of $\mo{F}$. Let $\widehat{\mo{F}}$ be $\mo{F}$ expanded by
interpreting each $P_j$ as a distinct upset of $\mo{F}$.
By Proposition~\ref{prop:om-satur-ext}
there exists an ultrapower $\widehat{\mo{F}}^I\!/U$
which is $\omega$-saturated. Now expanding (if
necessary) $\Prop$ to $\Prop^+$ so that there is a distinct
propositional variable
$p_j$ for each $P_j$, and reinterpreting $P_j$ as
the value of $p_j$, we have that $\widehat{\mo{F}}^I\!/U$
realises all sets of $\biLan(\Prop^+)$ formulas maximally consistent over $\mo{F}$.  
Define a map $f\colon \widehat{\mo{F}}^I\!/U \to \mathfrak{pe}(\mo{F})$ putting
$$
f(a) = \{P_j\in \mathrm{Up}(\mo{F}): \widehat{\mo{F}}^I\!/U\models P_j(a)\} 
$$
This map is order preserving, for if $a\leq b$, then since $P_j$ are upsets, we have
$f(a)\subseteq f(b)$. It also satisfies the bounded morphism conditions. For
suppose $f(a) \subseteq Z$, then since $Z$ is a prime-filter of upsets of $\mo{F}$, we
have $Z = \{P_j^F: j\in J, \text{ for some } J\}$, that is, $Z$ is a type.
Moreover, since $Z$ is a prime filter, $Z$ is maximally consistent and hence, by
saturation, $Z = f(b)$ for some $b$.
This proves the upward bounded-morphism condition, the proof of
the downward condition is analogous. The map $f$ does not depend on the language
expansion, so in fact it is a bounded morphism from
$\mo{F}^I\!/U$ to $\mathfrak{pe}(\mo{F})$.
\end{proof}

\begin{theorem}\label{thm:GT}
Let $K$ be a class of Kripke frames closed under ultrapowers. Then $K$ is
bi-intuitionistically definable if and only if $K$ is closed under
generated subframes, bounded-morphic images and disjoint unions, and reflects
prime-filter extensions.
\end{theorem}  

\begin{proof}
Assume $K$ is bi-intuitionistically definable.  Closure under generated
subframes, bounded-morphic images and 
disjoint unions follows immediately from Pro\-position~\ref{prop:truth-pres}.
By Proposition~\ref{inpe}(3) it also follows that $K$ reflects prime-filter 
extensions.

For the converse, assume $K$ is an elementary class closed under 
generated subframes, bounded-morphic images and disjoint unions, and reflecting
prime-filter extensions. Let $\Sigma(K)$ be the set of all
formulas valid in $K$. Take any frame $\mo{F}$ such that
$\mo{F}\Vdash\Sigma(K)$. To prove the theorem it suffices to show that $\mo{F}$
belongs to $K$. 

By Proposition~\ref{prop:algebraizability} we have
that $\mo{F}^+\models \{\varphi= 1: \varphi \in\Sigma(K)\}$.
By subdirect representation theorem,
$\mo{F}^+\in HSP\{\mo{P}^+: \mo{P}\in K\}$. Put
$\mathbf{A} = \mo{F}^+$, and $\mo{P} = \coprod_{i\in I} \mo{P}_i$.
Hence, there is a set
$\{\mo{P}_i: i\in I\}$ of frames from $K$ such that
$\mathbf{A}$ is a homomorphic image of an algebra
$\mathbf{B}$ which is a subalgebra of
$\mathbf{C} = \prod_{i\in I}\mo{P}_i^+$. By Proposition~\ref{prop:min-dual-frm}(1), 
$$
\mathbf{C} = \prod_{i\in I}\mo{P}_i^+\cong
\left(\coprod_{i\in I} \mo{P}_i\right)^+ = \mo{P}^+
$$
and since $K$ is closed under disjoint unions, we have
$\mo{P}\in K$. Using Proposition~~\ref{prop:min-dual-alg}(2,3), we obtain
the following: (i) $\mathbf{A}_+$ is a generated subframe of
$\mathbf{B}_+$, (ii) $\mathbf{B}_+$ is a bounded-morphic image of
$\mathbf{C}_+ \cong \left(\mo{P}^+\right)_+ = \mathfrak{pe}(\mo{P})$. 
Since $K$ is closed under ultrapowers, by Lemma~\ref{lem:satur} we get
that $\mathfrak{pe}(\mo{P})\in K$.
But then, $\mo{P}\in K$, so by closure properties of $K$ 
we get that $\mathbf{B}_+$ and $\mathbf{A}_+$ are in $K$.
Finally, since $\mathbf{A}_+ = (\mo{F}^+)_+ \cong \mathfrak{pe}(\mo{F})$, and
$K$ reflects prime-filter extensions, $\mo{F}\in K$, as required.
\end{proof}

\section{Conclusion}\label{con}

In this paper we have settled two questions. On the one hand, we have provided a
characterization of the classes of model in bi-intuitionistic logic that can be
axiomatised by a theory in the language of the logic (Theorem~\ref{axmod}). On
the other hand, we have obtained an analogous result for the classes of frames
closed under ultrapowers for which a similar thing can be done
(Theorem~\ref{thm:GT}).  These two new results contribute to building a fuller
picture of the expressivity of propositional bi-intuitionistic logic, adding to
the results in~\cite{badia,olk}.

\end{document}